\documentclass[11pt,letterpaper]{article}
\usepackage{amssymb,amsmath,graphicx,amsfonts}
\usepackage{amsmath}
\usepackage{amsfonts}
\usepackage{tikz}
\usetikzlibrary{arrows}
\usepackage{color}
\usepackage{xcolor}

\newtheorem{theorem}{Theorem}

\newtheorem{lemma}[theorem]{Lemma}

\newtheorem{corollary}[theorem]{Corollary} 

\newtheorem{example}[theorem]{Example}
\newtheorem{problem}[theorem]{Problem}

\newtheorem{definition}[theorem]{Definition}
\newtheorem{conjecture}[theorem]{Conjecture}
\newtheorem{preproof}{{\bf Proof}}

\newenvironment{proof}[1]{\begin{preproof}{\rm
			#1}\hfill{$\blacksquare$}}{\end{preproof}}
\newtheorem{presproof}{{\bf Sketch of Proof.\ }}

\newtheorem{prepro}{{\bf Proposition}}

\begin{document}

\title{Derangement Representation of Graphs}
{\small
	\author{Somayeh Ashofteh, Moharram N. Iradmusa\\
		{\small Department of Mathematical Sciences, Shahid Beheshti University,}\\
		{\small G.C., P.O. Box 19839-63113, Tehran, Iran.}}

\maketitle
\begin{abstract}
A derangement $k$-representation of a graph $G$  is a map $\pi$ of $V(G)$ to the symmetric group $S_k$, such that for any two vertices $v$ and $u$ of $V(G)$, $v $ and $u$ are adjacent if and only if $\pi(v)(i) \neq \pi(u)(i)$ for each $i \in \{1,2,3,\ldots,k\}$. The derangement representation number of $G$ denoted by  $drn(G)$, is the minimum of $k$ such that $G$ has a derangement $k$-representation. In this paper, we prove that any graph has a derangement $k$-representation. Also, we obtain some lower and upper bounds for $drn(G)$, in terms of the basic parameters of $G$. Finally, we determine the exact value or give the better bounds of the derangement representation number of some classes of graphs.
\end{abstract}

\section{Introduction}
All graphs we consider in this paper are simple, finite, and undirected. For a graph $G$, we denote its vertex set and edge set by $V(G)$ and $E(G)$, respectively. Also we use the notations $p(G)$, $q(G)$, $\omega(G)$ and $G^c$ for the order, the size, the maximum size of cliques and the complement graph of $G$, respectively. The path and the cycle of order $n$ are denoted by $P_n$ and $C_n$, respectively. From now on,  we use the notation $[n]$ and $\mathbb{N}_{\geq n}$ instead of $\{1,\ldots,n\}$ and $\{m\in\mathbb{N}|\ m\geq n\}$, respectively. We mention some definitions that are referred to throughout this paper and for other necessary definitions and notation we refer the reader to a standard text-book \cite{bondy}.\\
There are many geometric and algebraic representations of graphs, such as intersection graphs, interval graphs \cite{trotter1979double}, orthogonal latin square graphs \cite{lindner1979orthogonal}, and Cayley graphs. Formally, an intersection graph $G$ is a graph formed from a family of finite sets $\mathcal{F}=\{S_1, S_2,\ldots, S_n\}$, by creating one vertex $v_i$ for each set $S_i$, and connecting two vertices $v_i$ and $v_j$ by an edge whenever the corresponding two sets have a nonempty intersection, that is, $E(G)=\{\{v_i,v_j\}\mid i\neq j,\ S_{i}\cap S_{j}\neq\emptyset \}$.  The intersection number of a graph $G$ is the minimum total number of elements in a representation of $G$ as an intersection graph of finite sets. 
In 1966, Erdos et al. proved that the intersection number of $G$ is at most $\frac{n^2}{4}$, where $n$ is the order of $G$ \cite{erdos1966}.

Replacing the finite sets by open intervals, the interval number of a graph $G$, denoted by $i(G)$, is the minimum $t$ such that $G$ is the intersection graph of sets consisting of at most $t$ intervals on the real line \cite{trotter1979double}. Such a description of $G$ is called a $t$-interval representation of $G$.

An orthogonal latin square graph (OLSG) is one in which the vertices are latin squares of the same order and on the same symbols, and two vertices are adjacent if and only if the corresponding latin squares are orthogonal \cite{lindner1979orthogonal}. Erdos and Evans proved that any finite graph can be realized as an orthogonal latin square graph \cite{erdos1989representations}. 

A Cayley graph is a graph on a group $G$ with connection set $S$ satisfying $1 \notin S$ and $S=S^{-1}$,  denoted by $Cay(G,S)$, such that the vertices are the elements of $G$ and there is an edge joining $g$ and $h$ if and only if $h = sg$ for some $s \in S$ \cite{beineke2004topics}. This concept was introduced by Arthur Cayley in $1878$ \cite{cayley1854vii}. We know that any Cayley graph is a vertex-transitive graph and so the family of Cayley graphs is a proper subfamily of the family of all graphs. But Babai and Sos in a probabilistic approach showed that if $G$ is a finite graph, then every sufficiently large group has a Cayley graph containing an induced subgraph isomorphic to $G$, precisely if $X$ is a finite graph of order $n$ and $G$ is a group of order at least $c_1 n^3$, then $X$ is isomorphic to an induced subgraph of $Cay(G,S)$ for some $S\subseteq G$ \cite{babai1985sidon}. Furthermore, in group theory, the Cayley theorem states that every group $G$ is isomorphic to a subgroup of a symmetric group \cite{cayley1854vii}. Therefore the study of Cayley graphs on symmetric groups has great importance.

The main objective of this work is to provide an algebraic approach for finding a Cayley graph containing an induced subgraph that is isomorphic to any given graph $G$. To show that we first, introduce a new representation of graphs using derangements. A permutation on a finite set $X$ is a bijection on $X$ and the set of all permutations on  $X$ is denoted by $S_X$. When $X=[n]$, $S_X$ is usually denoted by $S_n$. We use the notation $(\pi_1,\ldots,\pi_n)$ for the permutation $\pi\in S_n$ where $\pi(i)=\pi_i$ for each $i\in [n]$. A derangement $\sigma\in S_n$ is a permutation that has no fixed points, which means $\sigma_i\neq i$ for all $i\in [n]$. The set of all derangements of $S_k$ is denoted by $D_k$.\\
\begin{definition}\label{def1}
Let $G$ be a graph and $k\in\mathbb{N}$. We say $G$ is a {\it derangement $k$-representable} if there exists an injective map $\pi:V(G)\rightarrow S_k$, such that for any two vertices $v$ and $u$ of $G$, $v$ and $u$ are adjacent if and only if $\pi(v)(i) \neq \pi(u)(i)$ for all $i \in [k]$. In other words,
$\pi(v)^{-1}\circ \pi(u) \in D_k$. In this case, $\pi$ is called a derangement $k$-representation of $G$. The derangement representation number of $G$, denoted by $drn(G)$, is the minimum of $k$ such that $G$ has a derangement $k$-representation.\\
\end{definition}
\begin{example}\label{example1}
Figure \ref{figP3} shows a derangement 4-representation of $P_3$ ($=K_3-K_2$). In addition, we prove that $P_3$ has no derangement 3-representation. Suppose that $\pi: V(P_3)\rightarrow S_3$ is a derangement 3-representation of $P_3$. So $\pi(v_1)^{-1}\circ \pi(v_2)\in D_3$ and $\pi(v_2)^{-1}\circ \pi(v_3)\in D_3$. In addition $D_3=\{\sigma,\sigma^{-1}\}$, where $\sigma=(2,3,1)$. Therefore, $\{\pi(v_1)^{-1}\circ\pi(v_2), \pi(v_2)^{-1}\circ \pi(v_3)\}=\{\sigma,\sigma^{-1}\}$. So we have $\pi(v_1)^{-1}\circ \pi(v_2)\circ \pi(v_2)^{-1}\circ \pi(v_3)=\pi(v_1)^{-1}\circ \pi(v_3)=1$ which concludes $\pi(v_1)=\pi(v_3)$, a contradiction. Then $drn(P_3)=4$.
\end{example}
\begin{figure}[h]
\label{figP3}
 	\begin{center}
 	\begin{tikzpicture}[scale=0.7]
 	\tikzset{vertex/.style = {shape=circle,draw, line width=1pt,opacity=1.0}}
  \node[vertex] (x) at (-3,0) {};
  \node[vertex] (y) at (0,0) {};
  \node[vertex] (z) at (3,0) {};
  \foreach \from/\to in {x/y,y/z}
  \draw[line width=1pt] (\from) -- (\to);
  \node  at (-3,-0.7) {\small{$(1,2,3,4)$}};
  \node  at (0,-0.7) {\small{$(3,4,1,2)$}};
  \node  at (3,-0.7) {\small{$(1,2,4,3)$}};
  \node  at (-3,0.7) {\small{$v_1$}};
  \node  at (0,0.7) {\small{$v_2$}};
  \node  at (3,0.7) {\small{$v_3$}};
		\end{tikzpicture}
 		\caption{A derangement 4-representation of $P_3$}
 		\label{figP3}
 		\end{center}
 \end{figure}
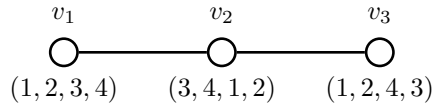
According to the Definition \ref{def1}, we can prove the following theorem which shows the relationship between Cayley graphs associated to the symmetric groups and derangement representation of graphs.\\
\begin{theorem}\label{thm2}
Let $G$ be a graph. Then $G$ is derangement $l$-representable if and only if $G$ is isomorphic to an induced subgraph of $\Gamma=Cay(S_l,D_l)$.
\end{theorem}
As said before, the main result of \cite{babai1985sidon} and Cayley theorem in group theory confirm that any graph has a deranement $l$-representation for some $l\in\mathbb{N}$. In the next theorems we will prove this fact by finding some general upper bounds for derangement representation number of graphs.\\

\begin{theorem}\label{thm3}
(1) $drn(K_n)=n$ for any $n\in\mathbb{N}$.\\
(2) $drn(K_n-K_2)=n$ for any positive integer $n\geq4$.\\
(3) Let $G$ be a graph of order $n$ and $q(G^c)\geq2$. Then $drn(G) \leq (n-1)q(G^c)$.\\
\end{theorem}
A {\it clique decomposition} of a graph $G$ is a collection of non-trivial cliques which partition the edge set of $G$.
\begin{theorem}\label{thm4}
Let $G$ be a graph of order $n$, $D=\{C_1,C_2,\ldots,C_s\}$ be a clique decomposition of $G^c$ and $s\geq2$. Then 
$$drn(G) \leq  s(n+1)- \sum_{i=1}^{s}p(C_i).$$
\end{theorem}
\subsection{Structure of the paper}
		After this introductory section where we established the background, purpose, and some basic definitions and theorems of the paper, we divide the paper into four sections. In Section \ref{sec2}, we prove Theorems \ref{thm2}, \ref{thm3} and \ref{thm4} and some basic lemmas and theorems and introduce the concept {\it derangement representation matrix} of graphs. In Section \ref{sec3}, we determine the exact value of the derangement representation number for some nearly complete graphs and give the better bounds of the parameter for some classes of graphs. In Section \ref{sec4}, we present some computations that performed by SageMath \cite{sagemath} and in the last section, we state some conjectures and open problems.
\section{Proofs of Thorems \ref{thm2}, \ref{thm3} and \ref{thm4}}\label{sec2}
At first, we prove Theorem \ref{thm2}.\\
		\textbf{Proof of Theorem \ref{thm2}}. Let $G$ be a derangement $l$-representable graph and $\pi:V(G)\rightarrow S_l$ be a derangement $l$-representation of $G$. Then for each vertex $v_i$, $\pi(v_i)$ is a vertex of $Cay(S_l,D_l)$ and for every two vertices $v_i$ and $v_j$ of $G$, we have $\pi(v_i)^{-1}\circ \pi(v_j) \in D_l$ if and only if $v_i$ and $v_j$ are adjacent. Therefore $\pi$ is an isomorphism from $G$ to the subgraph of $Cay(S_l,D_l)$ induced by $\pi(V(G))$.\\
		Now suppose that $G$ is isomorphic to an induced subgraph of $Cay(S_l,D_l)$ and $\pi$ is this isomorphism. Easily one can show that $\pi$ is a derangement $l$-representation of $G$.
				\hfill $\square$\\
\begin{lemma}\label{inducedsubgraph}
Let $G$ be a graph and $H$ be an induced subgraph of $G$. Then $drn(G)\geq drn(H)$.
\end{lemma}
\begin{proof}
{Suppose that $drn(G)=k$. Hence $G$ is isomorphic to an induced subgraph of $Cay(S_k,D_k)$ such as $\Gamma$. Let $\alpha:V(G)\rightarrow V(\Gamma)$ be an isomorphism. Therefore, the subgraph of $\Gamma$ induced by $\alpha(V(H))$ is isomorphic to $H$. So $H$ is isomorphic to an induced subgraph of $Cay(S_k,D_k)$. Thus $H$ is derangement $k$-representable and so $drn(H)\leq k$.
}\end{proof}
In the next definition, we introduce an equivalent representation of graphs using matrices.
\begin{definition}\label{matrixrep}
Let $G$ be a graph with $V(G)=\{v_1,v_2,\ldots,v_n\}$ and $\pi$ is a derangement $k$-representation of $G$. A derangement $k$-representation matrix of $G$ associated to $\pi$ is defined as follows:
\[L(\pi,G)=
\begin{bmatrix}
   \pi(v_1)(1) & \pi(v_1)(2)  & \cdots & \pi(v_1)(k) \\ 
   \pi(v_2)(1) & \pi(v_2)(2)  & \cdots & \pi(v_2)(k) \\ 
   \vdots & \vdots & \ddots & \vdots \\
   \pi(v_n)(1) & \pi(v_n)(2) & \cdots & \pi(v_n)(k)
\end{bmatrix}
\]
\end{definition}
The next definition is used in the proof of Theorem \ref{thm3}.
\begin{definition}
Let $L=[l_{i,j}]$ be a latin square of order $n-1$ and $1\leq r<s\leq n$. Then\\
$L(r,s)=[l'_{i,j}]_{n\times (n-1)}$, where
\[l'_{i,j}=\begin{cases}
 l_{i,j}&i<s\\ l_{r,j} & i=s\\ l_{(i-1),j}&i>s
\end{cases}.\]
Note that in $L(r,s)$, $l'_{i,j}\neq l'_{i',j}$ for all $j\in[n-1]$ when $\{r,s\}\neq\{i,i'\}$ and $l'_{i,j}= l'_{i',j}$ for all $j\in[n-1]$ when $\{r,s\}=\{i,i'\}$.
\end{definition}
\textbf{Proof of Theorem \ref{thm3}}. (1) Any latin square of order $n$ is a derangement $n$-representation matrix of $K_n$. So $drn(K_n)\leq n$. Now suppose that $\pi:V(K_n)\rightarrow S_k$ is a derangement $k$-representation of $K_n$. If $i\neq j$ then $\pi(v_i)(1)\neq \pi(v_j)(1)$. Therefore,
$|\{\pi(v_i)(1)| i\in[n]\}|=n$. Hence we have $k\geq n$ and so $drn(K_n)=n$.\\
(2) Let $G=K_n-K_2$, $n\geq4$ and  $V(G)=\{v_1,\ldots,v_n\}$. Suppose that $v_{1} $ and $v_2$ are two non-adjacent vertices. Since $G-v_1$ is a clique of $G$, $drn(G)\geq drn(G-v_1)=n-1$. Now we choose a latin square $L$ of order $n$ such that the first two rows of $L$ are $R_1=[1\ 2\ 3\ 4\ \cdots\ n]$ and $R_2=[2\ 1\ 4\ 5\ \cdots \ n\ 3]$. Then we replace $R_2$ with $R_2'=[1\ 2\ 4\ 5\ \cdots \ n\ 3]$. The resulting matix is a derangement $n$-representation matrix of $K_n$ ans so $drn(K_n)\leq n$. We prove that $drn(G)>n-1$. In contrary, suppose that $A=[a_{i,j}]_{n\times(n-1)}$ is a derangement $(n-1)$-representation matrix of $K_n-K_2$ where the corresponding vertex to $R_i$ (the $i$-th row of $A$) is $v_i$ and $v_1\nsim v_2$. The subgraph $K_n-v_1$ and $K_n-v_2$ are two cliques of order $n-1$. So the resulting matrices obtained from $A$ by removing any row of $\{R_1,R_2\}$ are latin squares of order $n-1$. In these latin squares, the last $n-2$ rows are the same rows, implying that $R_1=R_2$, a contradiction. Therefore $drn(K_n-K_2)=n$.\\
(3) Suppose that $N=(n-1)q(G^c)$, $V(G)=\{v_1,\ldots,v_n\}$ and
$E(G^c)=\{\{v_{i_t},v_{j_t}\}| 1\leq t\leq m, i_t<j_t\}$. Also suppose that $L^{(1)},\ldots,L^{(m)}$ are latin squares of order $n-1$  with mutually distinct symbols, such that the union set of all symbols is $[N]$. Now consider the following block matrix:
\[L=\bigg[\begin{array}{c|c|c|c}
   L^{(1)}(i_1,j_1) & L^{(2)}(i_2,j_2)&\cdots & L^{(m)}(i_m,j_m)
   \end{array}\bigg].\]
   We show that $L$ is a derangement $N$-representation matrix of $G$ associated to the map $\alpha:V(G)\rightarrow S_N$, where $\alpha(v_i)$ is $i$-th row of the matrix $L$. Suppose that $v_i$ and $v_{i'}$ are adjacent. Then $\{i,i'\}\notin\{\{i_t,j_t\}| 1\leq t\leq m, i_t<j_t\}$ and so $l'^{(k)}_{i,j}\neq l'^{(k)}_{i',j}$ for all $j\in[n-1]$ and all $k\in [m]$. Therefore $\alpha(v_i)(j)\neq \alpha(v_{i'})(j)$ for all $j\in [N]$. Now suppose that $v_i$ and $v_{i'}$ are not adjacent. Then $\{i,i'\}=\{i_{k_0},j_{k_0}\}$ for some $k_0\in [m]$ and so $l'^{(k_0)}_{i,j}= l'^{(k_0)}_{i',j}$ for all $j\in[n-1]$. Therefore $\alpha(v_i)(j_0)=\alpha(v_{i'})(j_0)$ for some $j_0\in [N]$. In addition, since $m\geq2$, $L$ has $n$ different rows and so $\alpha$ is an injective map. Therefore $drn(G)\leq N$. 
				\hfill $\square$\\
Applying Lemma \ref{inducedsubgraph} and Theorem \ref{thm3}, we conclude the following corollary.
\begin{corollary}
Let $G$ be a graph with clique number $\omega(G)$. Then $drn(G)\geq \omega(G)$.
\end{corollary}
\begin{definition}
Let $L=[l_{i,j}]$ be a latin square of order $n-k+1$, $S=\{s_1,\ldots,s_k\}$ and $1\leq s_1<\cdots<s_k\leq n$. Then $L(S)=[l'_{i,j}]_{n\times (n-k+1)}$, where
\[l'_{i,j}=\begin{cases}
l_{s_1,j} & i\in\{s_2,\ldots,s_k\}\\l_{i,j}&i<s_2 \\ l_{(i-1),j}&s_2<i<s_3\\
 l_{(i-2),j}&s_3<i<s_4\\ \cdots&\cdots \\l_{(i-k+2),j}&s_{k-1}<i<s_{k}\\l_{(i-k+1),j}&s_{k}<i
\end{cases}.\]
Note that in $L(S)$, $l'_{i,j}\neq l'_{i',j}$ for all $j\in[n-1]$ when $\{i,i'\}\nsubseteq S$ and $l'_{i,j}= l'_{i',j}$ for all $j\in[n-1]$ when $\{i,i'\}\subseteq S$.
\end{definition}
\textbf{Proof of Theorem \ref{thm4}}. For any $i \in[s]$ let $N_i=n-p_i+1$ and $L^{(i)}$ be a latin square of order $N_i$ where $p_i=p(C_i)$. Suppose that $A_i=V(C_i)=\{v^{i}_{j_1},\ldots,v^{i}_{j_{p_i}}\}$. Then we have $s$ latin squares $L^{(1)},\ldots,L^{(s)}$, with mutually distinct symbols such that the union set of all symbols is $\{1,2,\ldots,N\}$ where $N=\sum_{i=1}^{s}N_i=s(n+1)-\sum_{i=1}^{s}p_i$. Now consider the following block matrix:
\[L=\bigg[\begin{array}{c|c|c|c}
   L^{(1)}(A_1) & L^{(2)}(A_2)&\cdots & L^{(s)}(A_s)
   \end{array}\bigg],\]
in which $L^{(i)}(A_i)$ is a matrix of order $n\times(n-p_i+1)$. We show that $L$ is a derangement $N$-representation matrix of $G$ associated to the map $\alpha:V(G)\rightarrow S_N$, where $\alpha(v_i)$ is $i$-th row of the matrix $L$. Suppose that $v_i$ and $v_{i'}$ are adjacent. Then $\{i,i'\}\nsubseteq A_t$ for all $t\in[s]$ and so $l'^{(k)}_{i,j}\neq l'^{(k)}_{i',j}$ for all $j\in[n-p_k+1]$ and all $k\in [s]$. Therefore $\alpha(v_i)(j)\neq \alpha(v_{i'})(j)$ for all $j\in [N]$. Now suppose that $v_i$ and $v_{i'}$ are not adjacent. Then $\{i,i'\}\subseteq A_{k_0}$ for some $k_0\in [s]$ and so $l'^{(k_0)}_{i,j}= l'^{(k_0)}_{i',j}$ for all $j\in[n-p_{k_0}+1]$. Therefore $\alpha(v_i)(j_0)=\alpha(v_{i'})(j_0)$ for some $j_0\in [N]$. In addition, since $s\geq2$, $L$ has $n$ different rows and so $\alpha$ is an injective map. Therefore $drn(G)\leq N$.
\hfill $\square$
\section{Improved results for some classes of graphs}\label{sec3}
In this section, we will calculate some better lower and upper bounds or the exact amount of derangement representation numbers for specific families of graphs.\\
To prove the next theorem, we need the definition of {\it intersecting family of permutations}.
\begin{definition}\label{intersectingfamily}
A subset $S$ of $S_n$ is intersecting if for any two permutations $g$ and $h$ in $S$, $g(i)=h(i)$ for some $i\in [n]$ or equivalently $h^{-1}\circ g\notin D_n$.
\end{definition}
In \cite{deza1977}, M. Deza and P. Frankl proved that $|S|\leq (n-1)!$ for any intersecting family $S$ of permutations of symmetric group $S_n$.
\begin{theorem}\label{thm11}
Let $n,k\in\mathbb{N}$ and $(k-1)!<n\leq k!$. Then $drn(\overline{K_n})=k+1$. 
\end{theorem}
\begin{proof}{
Suppose that $drn((\overline{K_n})=t$ and $\pi:V(\overline{K_n})\rightarrow S_t$ is a derangement $t$-representation of $\overline{K_n}$. Therefore $\pi(V(\overline{K_n}))$ is an intersecting subset of $S_t$ and so $|\pi(V(\overline{K_n}))|=n\leq (t-1)!$. Thus $t-1\geq k$ and hence $drn((\overline{K_n})=t\geq k+1$.\\
To complete the proof, we give a derangement $(k+1)$-representation matrix of $\overline{K_n}$. Let $A=[a_{i,j}]_{n\times k}$ be a matrix that its rows represent $n$ permutations of $S_k$. Then the following matrix is a derangement $(k+1)$-representation matrix of $\overline{K_n}$:
\[L(\overline{K_n})=\left[\begin{array}{c|c}
\begin{array}{c}
   k+1 \\
   \vdots \\
    k+1  
   \end{array}& A  \end{array}\right].\]}
   \end{proof}
we know that $drn(K_n)=drn(K_n-K_2)=n$. Now we investigate to find the derangement representation number of nearly complete graphs. A graph is {\it nearly complete} if it can be obtained by removing a small number of edges from a complete graph relative to the size of the graph.
\begin{theorem}\label{p3}
(1) $drn(K_n-P_3)=n$ when $n\in\{3,4\}$, and\\ 
(2) $drn(K_n-P_3)= n-1$ when $n\in\mathbb{N}_{\geq5}$.
\end{theorem}
\begin{proof}{
(1) $\overline{K_2}$ is an induced subgraph of $K_3-P_3$. So $drn(K_3-P_3)\geq drn(\overline{K_2})=3$ by Theorem \ref{thm11} and Lemma \ref{inducedsubgraph}. In addition, the following matrix is a derangement $3$-representation matrix of $K_3-P_3$:
\[L(K_3-P_3)=\left[\begin{array}{ccc}
   {\bf1} & 2 & 3 \\
   2 & {\bf3} & 1 \\
   {\bf1} & {\bf3} & 2 \\
   \end{array}\right].\]
   Therefore, $drn(K_3-P_3)=3$.\\
$P_3$ is an induced subgraph of $K_4-P_3$. So $drn(K_4-P_3)\geq drn(P_3)=4$ by Lemma \ref{inducedsubgraph} and Example \ref{example1}. In addition, the following matrix is a derangement $4$-representation matrix of $K_4-P_3$:
\[L(K_3-P_3)=\left[\begin{array}{cccc}
   {\bf1} & {\bf2} & 3 & 4\\
   4 & 1 & 2 & {\bf3}\\
   3 & 4 & 1 & 2 \\
   {\bf1} & {\bf2} & 4 & {\bf 3} \\
   \end{array}\right].\]
   Therefore, $drn(K_4-P_3)=4$.\\
(2) $K_{n-1}$ is an induced subgraph of $K_n-P_3$. So $drn(K_n-P_3)\geq drn(K_{n-1})=n-1$ by Lemma \ref{inducedsubgraph}. Suppose that $A$ is a latin square of order $n-1$ such that the first two rows of $A$ are
$R_1=\left[\begin{array}{ccccccc} 1 & 2 & 3 & 4 & 5 & \cdots & n-1\\ \end{array}\right]$ and $R_2=\left[\begin{array}{ccccccc} 2 & 1 & n-1 & 3 & 4 & \cdots & n-2\\ \end{array}\right]$. Now the following matrix is a derangement $(n-1)$-representation matrix of $K_n-P_3$:
\[L(K_n-P_3)=\left[\begin{array}{c}
A \\
\hline
\begin{array}{ccccccc} 1 & 2 & n-1 & 3 & 4 & \cdots & n-2\\ \end{array} 
\end{array}\right].\]
Therefore, $drn(K_n-P_3)=n-1$.
}\end{proof}
\begin{theorem}
(1) $drn(K_n-2K_2)=n$ when $n\in\{4,5,6\}$, and\\ 
(2) $drn(K_n-2K_2)= n-1$ when $n\in\mathbb{N}_{\geq7}$.
\end{theorem}
\begin{proof}{
(1) $P_3$ is an induced subgraph of $K_4-2K_2$. So $drn(K_3-2K_2)\geq drn(P_3)=4$ by Lemma \ref{inducedsubgraph}. In addition, the following matrix is a derangement $4$-representation matrix of $K_4-2K_2$:
\[L(K_4-2K_2)=\left[\begin{array}{cccc}
   {\bf1} & {\bf2} & 3 & 4 \\
   {\bf1} & {\bf2} & 4 & 3 \\
   3 & 4 & {\bf1} & {\bf2} \\
   4 & 3 & {\bf1} & {\bf2} \\
   \end{array}\right].\]
   Therefore, $drn(K_4-2K_2)=4$.\\
The following matrix is a derangement $5$-representation matrix of $K_5-2K_2$:
\[L(K_4-2K_2)=\left[\begin{array}{ccccc}
   5 & {\bf2} & 3 & 4 & 1 \\
   1 & {\bf2} & 4 & 5 & 3 \\
   {\bf3} & 5 & 1 & 2 & 4 \\
   {\bf3} & 4 & 5 & 1 & 2 \\
   4 & 1 & 2 & 3 & 5 \\
   \end{array}\right].\]
   Therefore, $drn(K_5-2K_2)\leq5$. We show that $drn(K_5-2K_2)>4$. In contrary, suppose that $A=[a_{i,j}]_{5\times4}$ is a derangement $4$-representation matrix of $K_5-2K_2$ where the corresponding vertex to the $i$-th row is $v_i$, $v_1\nsim v_2$ and $v_3\nsim v_4$. Therefore  we have $a_{1,j}=a_{2,j}$ or $a_{3,j}=a_{4,j}$ in each column $C_j$ of $A$. In addition $1\leq |\{j\in[4]|\ a_{1,j}=a_{2,j}\}|\leq2$ and $1\leq |\{j\in[4]|\ a_{3,j}=a_{4,j}\}|\leq2$. So $|\{j\in[4]|\ a_{1,j}=a_{2,j}\}|=|\{j\in[4]|\ a_{3,j}=a_{4,j}\}|=2$ and $\{j\in[4]|\ a_{1,j}=a_{2,j}\}\cap\{j\in[4]|\ a_{3,j}=a_{4,j}\}=\emptyset $. Without loss of generality, assume that
\[A=\left[\begin{array}{ccccc}
   {\bf a} & {\bf b} & c & d \\
   {\bf a} & {\bf b} & d & c \\
   c & d & {\bf a} & {\bf b} \\
   d & c & {\bf a} & {\bf b} \\
   a_{5,1} & a_{5,2} & a_{5,3} & a_{5,4} \\
   \end{array}\right].\]
   So $a_{5,1}=a_{5,3}=b$ and $a_{5,2}=a_{5,4}=a$, a contradiction. Therefore, $drn(K_5-2K_2)=5$.\\
   The following matrix is a derangement $6$-representation matrix of $K_6-2K_2$:
\[L(K_6-2K_2)=\left[\begin{array}{cccccc}
   {\bf2} & 6 & {\bf4} & 1 & 3 & {\bf5} \\
   {\bf2} & 1 & {\bf4} & 3 & 6 & {\bf5} \\
   {\bf3} & 4 & 5 & {\bf6} & 1 & {\bf2} \\
   {\bf3} & 5 & 1 & {\bf6} & 4 & {\bf2} \\
   1 & 2 & 3 & 4 & 5 & 6 \\
   4 & 3 & 6 & 5 & 2 & 1 \\
   \end{array}\right].\]
   Therefore, $drn(K_6-2K_2)\leq6$. We show that $drn(K_6-2K_2)>5$. In contrary, suppose that $A=[a_{i,j}]_{6\times5}$ is a derangement $5$-representation matrix of $K_6-2K_2$ where the corresponding vertex to the $i$-th row is $v_i$, $v_1\nsim v_2$ and $v_3\nsim v_4$. Therefore we have $a_{1,j}=a_{2,j}$ or $a_{3,j}=a_{4,j}$ in each column $C_j$ of $A$. In addition $1\leq |\{j\in[4]|\ a_{1,j}=a_{2,j}\}|\leq3$ and $1\leq |\{j\in[4]|\ a_{3,j}=a_{4,j}\}|\leq3$. So there are three cases:\\
{\bf Case 1.} $|\{j\in[4]|\ a_{1,j}=a_{2,j}\}|=|\{j\in[4]|\ a_{3,j}=a_{4,j}\}|=3$ and $\{j\in[4]|\ a_{1,j}=a_{2,j}\}\cap\{j\in[4]|\ a_{3,j}=a_{4,j}\}=\{j_0\}$. Without loss of generality, assume that $j_0=3$, $\{j\in[4]|\ a_{1,j}=a_{2,j}\}=\{1,2,3\}$ and $\{j\in[4]|\ a_{3,j}=a_{4,j}\}=\{3,4,5\}$ and
\[A=\left[\begin{array}{cccccc}
   {\bf a} & {\bf b} & {\bf c} & d & e \\
   {\bf a} & {\bf b} & {\bf c} & e & d \\
   a_{3,1} & a_{3,2} & {\bf x} & {\bf y} & {\bf z} \\
   a_{4,1} & a_{4,2} & {\bf x} & {\bf y} & {\bf z} \\
   a_{5,1} & a_{5,2} & a_{5,3} & a_{5,4} & a_{5,5} \\
    a_{6,1} & a_{6,2} & a_{6,3} & a_{6,4} & a_{6,5} \\
   \end{array}\right].\]
   In this case, we have $y,z\notin\{d,e\}$ and hence $\{d,e\}\setminus\{x,y,z\}\neq\emptyset$. Therefore at least one of numbers $\{d, e\}$ must apear in the cells $\{a_{3,1}, a_{4,2}\}$ or $\{a_{4,1}, a_{3,2}\}$. But in this case, that number must apear in the cells $\{a_{5,3}, a_{6,3}\}$, a contradiction.\\
   {\bf Case 2.} $|\{j\in[4]|\ a_{1,j}=a_{2,j}\}|=2$ and $|\{j\in[4]|\ a_{3,j}=a_{4,j}\}|=3$ and $\{j\in[4]|\ a_{1,j}=a_{2,j}\}\cap\{j\in[4]|\ a_{3,j}=a_{4,j}\}=\emptyset$. Without loss of generality, assume that $\{j\in[4]|\ a_{1,j}=a_{2,j}\}=\{1,2\}$ and $\{j\in[4]|\ a_{3,j}=a_{4,j}\}=\{3,4,5\}$ and
\[A=\left[\begin{array}{cccccc}
   {\bf a} & {\bf b} & c & d & e \\
   {\bf a} & {\bf b} & d & e & c \\
   a_{3,1} & a_{3,2} & {\bf x} & {\bf y} & {\bf z} \\
   a_{4,1} & a_{4,2} & {\bf x} & {\bf y} & {\bf z} \\
   a_{5,1} & a_{5,2} & a_{5,3} & a_{5,4} & a_{5,5} \\
    a_{6,1} & a_{6,2} & a_{6,3} & a_{6,4} & a_{6,5} \\
   \end{array}\right].\]
   In this case, we have $|\{x,y,z\}\cap\{c,d,e\}|=1$ and hence $|\{c,d,e\}\setminus\{x,y,z\}|=2$. Therefore two numbers of $\{c, d, e\}$ must apear in the cells $\{a_{3,1}, a_{4,2}, a_{4,1}, a_{3,2}\}$. But in this case, that numbers must apear in the cells $\{a_{5,3}, a_{6,3}\}$, a contradiction.\\
   {\bf Case 3.} $|\{j\in[4]|\ a_{1,j}=a_{2,j}\}|=3$ and $|\{j\in[4]|\ a_{3,j}=a_{4,j}\}|=2$ and $\{j\in[4]|\ a_{1,j}=a_{2,j}\}\cap\{j\in[4]|\ a_{3,j}=a_{4,j}\}=\emptyset$. This case is similar to the previous case.\\
Therefore, $drn(K_6-2K_2)=6$.\\
(2) $K_{n-1}-K_2$ is an induced subgraph of $K_n-2K_2$. So $drn(K_n-2K_2)\geq drn(K_{n-1}-K_2)=n-1$ by Lemma \ref{inducedsubgraph} and Theorem \ref{thm3}. Suppose that $A$ is a latin square of order $n-1$ such that the first three rows of $A$ are $R_1=\left[\begin{array}{ccccccccc} 1 & 2 & 3 & | & 4 & 5 & \cdots & n-2 & n-1\\ \end{array}\right]$ (corresponding to $v_1$), $R_2=\left[\begin{array}{ccccccccc} 3 & 1 & 2 & | & 5 & 6 & \cdots & n-1 & 4 \\ \end{array}\right]$ (corresponding to $v_2$), and $R_3=\left[\begin{array}{ccccccccc} 2 & 3 & 1 & | & 6 & 7 & \cdots & 4 & 5 \\ \end{array}\right]$ (corresponding to $v_3$). At first, we replace the second row of $A$ with
\[R'_2=\left[\begin{array}{ccccccccc} 1 & 2 & 3 & | & 5 & 6 & \cdots & n-1 & 4 \\ \end{array}\right]\]
to achieve a derangement $(n-1)$-representation matrix $A'$ of $K_{n-1}-K_2$ in which $v_1$ and $v_2$ are not adjacent. Now the following matrix is a derangement $(n-1)$-representation matrix of $K_n-2K_2$:
\[L(K_n-2K_2)=\left[\begin{array}{c}
A' \\
\hline
\begin{array}{ccccccccc} 3 & 1 & 2 & | & 6 & 7 & \cdots & 4 & 5 \\ \end{array} 
\end{array}\right],\]
where the second part of the $n$-th row is the second part of the third row and the corresponding vertex to the last row is not adjacent to $v_3$. Therefore, $drn(K_n-2K_2)=n-1$.
}\end{proof}
\begin{theorem}\label{k3}
(1) $drn(K_n-K_3)=n$ when $n\in\{4,5,6\}$, and\\ 
(2) $drn(K_n-K_3)= n-1$ when $n\in\mathbb{N}_{\geq7}$.
\end{theorem}
\begin{proof}{
(1) $P_3$ is an induced subgraph of $K_4-K_3$. So $drn(K_4-K_3)\geq drn(P_3)=4$ by Lemma \ref{inducedsubgraph}. In addition, the following matrix is a derangement $4$-representation matrix of $K_4-2K_2$:
\[L(K_4-2K_2)=\left[\begin{array}{cccc}
   {\bf1} & {\bf4} & 2 & 3 \\
   {\bf1} & {\bf4} & {\bf3} & 2 \\
   {\bf1} & 2 & {\bf3} & 4 \\
   2 & 3 & 4 & 1 \\
   \end{array}\right].\]
   Therefore, $drn(K_4-K_3)=4$.\\
The following matrix is a derangement $5$-representation matrix of $K_5-K_3$:
\[L(K_5-K_3)=\left[\begin{array}{ccccc}
   {\bf1} & {\bf2} & 3 & 4 & {\bf5} \\
   {\bf1} & 4 & {\bf3} & 2 & {\bf5} \\
   4 & {\bf2} & {\bf3} & 1 & {\bf5} \\
   2 & 3 & 5 & 4 & 1 \\
   3 & 1 & 2 & 5 & 4 \\
   \end{array}\right].\]
   Therefore, $drn(K_5-K_3)\leq5$. We show that $drn(K_5-K_3)>4$. In contrary, suppose that $A=[a_{i,j}]_{5\times4}$ is a derangement $4$-representation matrix of $K_5-K_3$ where the corresponding vertex to the $i$-th row is $v_i$, $v_1\nsim v_2$, $v_1\nsim v_3$ and $v_2\nsim v_3$. Therefore we have $a_{1,j}=a_{2,j}$, $a_{2,j}=a_{3,j}$ or $a_{3,j}=a_{1,j}$ in each column $C_j$ of $A$. Let $a^{i,j}=|\{k\in[4]|\ a_{i,k}=a_{j,k}\}|$. So $1\leq a^{i,j}\leq2$ for any $\{i,j\}\subset\{1,2,3\}$ and $a^{1,2}+a^{2,3}+a^{3,1}\geq4$. So $a^{i,j}=2$ for some $\{i,j\}\subset\{1,2,3\}$. Without loss of generality, assume that $a^{1,2}=2$ and
\[A=\left[\begin{array}{ccccc}
   {\bf a} & {\bf b} & c & d \\
   {\bf a} & {\bf b} & d & c \\
   a_{3,1} & a_{3,2} & a_{3,3} & a_{3,4} \\
   a_{4,1} & a_{4,2} & a_{4,3} & a_{4,4} \\
   a_{5,1} & a_{5,2} & a_{5,3} & a_{5,4} \\
   \end{array}\right].\]
   If $(a_{3,3}, a_{3,4})=(a_{1,3},a_{2,4})$ or $(a_{1,4},a_{2,3})$ then $a_{3,3}=a_{3,4}$, a contradiction. If $(a_{3,3}, a_{3,4})=(c,d)$, then $a_{3,1}=a$ or $a_{3,2}=b$ which force the equality of the first and the third rows, a contradiction. Similarly, we have a contradiction when $(a_{3,3}, a_{3,4})=(d,c)$. Therefore, $drn(K_5-K_3)=5$.\\
   The following matrix is a derangement $6$-representation matrix of $K_6-K_3$:
\[L(K_6-K_3)=\left[\begin{array}{cccccc}
   {\bf1} & {\bf2} & 3 & 4 & {\bf5} & {\bf6} \\
   {\bf1} & {\bf2} & {\bf4} & {\bf3} & 6 & 5 \\
   2 & 1 & {\bf4} & {\bf3} & {\bf5} & {\bf6} \\
   3 & 4 & 5 & 6 & 1 & 2 \\
   4 & 3 & 6 & 5 & 2 & 1 \\
   5 & 6 & 1 & 2 & 3 & 4 \\
   \end{array}\right].\]
   Therefore, $drn(K_6-K_3)\leq6$. We show that $drn(K_6-K_3)>5$. In contrary, suppose that $A=[a_{i,j}]_{6\times5}$ is a derangement $5$-representation matrix of $K_6-K_3$ where the corresponding vertex to the $i$-th row is $v_i$, $v_1\nsim v_2$, $v_2\nsim v_3$ and $v_3\nsim v_1$. Therefore we have $a_{1,j}=a_{2,j}$, $a_{2,j}=a_{3,j}$ or $a_{3,j}=a_{1,j}$ in each column $C_j$ of $A$. Let $a^{i,j}=|\{k\in[4]|\ a_{i,k}=a_{j,k}\}|$. So $1\leq a^{i,j}\leq3$ for any $\{i,j\}\subset\{1,2,3\}$ and $a^{1,2}+a^{2,3}+a^{3,1}\geq5$. So there are two cases:\\
{\bf Case 1.} $a^{i,j}=3$ for some $\{i,j\}\subset\{1,2,3\}$. Without loss of generality, assume that $a^{1,2}=3$ and $\{j\in[4]|\ a_{1,j}=a_{2,j}\}=\{1,2,3\}$. Therefore,
\[A=\left[\begin{array}{cccccc}
   {\bf a} & {\bf b} & {\bf c} & d & e \\
   {\bf a} & {\bf b} & {\bf c} & e & d \\
   a_{3,1} & a_{3,2} & a_{3,3} & a_{3,4} & a_{3,5} \\
   a_{4,1} & a_{4,2} & a_{4,3} & a_{4,4} & a_{4,5} \\
   a_{5,1} & a_{5,2} & a_{5,3} & a_{5,4} & a_{5,5} \\
    a_{6,1} & a_{6,2} & a_{6,3} & a_{6,4} & a_{6,5} \\
   \end{array}\right].\]
 If $(a_{3,4}, a_{3,5})=(a_{1,4},a_{2,5})$ or $(a_{1,5},a_{2,4})$ then $a_{3,4}=a_{3,5}$, a contradiction. If $(a_{3,4}, a_{3,5})\in\{(d,e), (e,d)\}$, then $(a_{3,1},a_{3,2},a_{3,3})\in\{(a,c,b), (c,b,a), (b,a,c)\}$. Without loss of generality, assume that $(a_{3,4}, a_{3,5})=(d,e)$ and $(a_{3,1},a_{3,2},a_{3,3})=(a,c,b)$. Then
 \[A=\left[\begin{array}{cccccc}
   {\bf a} & {\bf b} & {\bf c} & {\bf d} & {\bf e} \\
   {\bf a} & {\bf b} & {\bf c} & e & d \\
   {\bf a} & c & b & {\bf d} & {\bf e} \\
   a_{4,1} & a_{4,2} & a_{4,3} & a_{4,4} & a_{4,5} \\
   a_{5,1} & a_{5,2} & a_{5,3} & a_{5,4} & a_{5,5} \\
    a_{6,1} & a_{6,2} & a_{6,3} & a_{6,4} & a_{6,5} \\
   \end{array}\right].\]
 But in this case we have $\{b,c,d,e\}\subset\{a_{4,1},a_{5,1},a_{6,1}\}$, a contradiction.\\
{\bf Case 2.} $a^{i,j}=2$ and $a^{j,k}=2$ where $\{i,j,k\}=[3]$. Without loss of generality, assume that $a^{1,2}=a^{2,3}=2$ and $\{j\in[5]|\ a_{1,j}=a_{2,j}\}=\{1,2\}$ and $\{j\in[5]|\ a_{2,j}=a_{3,j}\}=\{3,4\}$ and
\[A=\left[\begin{array}{cccccc}
   {\bf a} & {\bf b} & c & d & e \\
   {\bf a} & {\bf b} & {\bf d} & {\bf e} & c \\
   a_{3,1} & a_{3,2} & {\bf d} & {\bf e} & a_{3,5} \\
   a_{4,1} & a_{4,2} & a_{4,3} & a_{4,4} & a_{4,5} \\
   a_{5,1} & a_{5,2} & a_{5,3} & a_{5,4} & a_{5,5} \\
    a_{6,1} & a_{6,2} & a_{6,3} & a_{6,4} & a_{6,5} \\
   \end{array}\right].\]
   Since $\{a_{1,5},a_{2,5},a_{3,5}\}\cap\{a_{4,5},a_{5,5},a_{6,5}\}=\emptyset$ and $\{a_{1,5},a_{2,5},a_{3,5}\}\cup\{a_{4,5},a_{5,5},a_{6,5}\}\subseteq [5]$  hence $a_{3,5}\in\{c,e\}$. But $a_{3,5}\neq a_{3,4}=e$ and so $a_{3,5}=c$, which contradicts by $a^{2,3}=2$.\\
Therefore, $drn(K_6-K_3)=6$.\\
(2) $K_{n-2}$ is an induced subgraph of $K_n-K_3$. So $drn(K_n-K_3)\geq drn(K_{n-2})=n-2$ by Lemma \ref{inducedsubgraph} and Theorem \ref{thm3}. We show that $drn(K_n-K_3)>n-2$. In contrary, suppose that $A=[a_{i,j}]_{n\times(n-2)}$ is a derangement $(n-2)$-representation matrix of $K_n-K_3$ where the corresponding vertex to $R_i$ (the $i$-th row of $A$) is $v_i$, $v_1\nsim v_2$, $v_2\nsim v_3$ and $v_3\nsim v_1$. The subgraphs induced by removing any two vertices of $\{v_1,v_2,v_3\}$ are cliques of order $n-2$. So the resulting matrices obtained from $A$ by removing any two rows of $\{R_1,R_2,R_3\}$ are latin squares of order $n-2$. In these latin squares, the last $n-3$ rows are the same rows, implying that $R_1=R_2=R_3$, a contradiction.\\
Suppose that $A$ is a latin square of order $n-1$ such that the first two rows of $A$ are
\[R_1=\left[\begin{array}{ccccccccccc} 1 & 2 & | & 3 & 4 & | & 5 & 6 & \cdots & n-2 & n-1\\ \end{array}\right]\]
(corresponding to $v_1$) and $R_2=\left[\begin{array}{ccccccccccc} 2 & 1 & | & 4 & 3 & | & 6 & 7 & \cdots & n-1 & 5 \\ \end{array}\right]$ (corresponding to $v_2$). At first, we replace the second row of $A$ by $R'_2=\left[\begin{array}{ccccccccccc} 1 & 2 & | & 4 & 3 & | & 6 & 7 & \cdots & n-1 & 5 \\ \end{array}\right]$ to achieve a derangement $(n-1)$-representation matrix $A'$ of $K_{n-1}-K_2$ in which $v_1$ and $v_2$ are not adjacent. Now the following matrix is a derangement $(n-1)$-representation matrix of $K_n-K_3$:
\[L(K_n-K_3)=\left[\begin{array}{c}
A' \\
\hline
\begin{array}{ccccccccccc} 2 & 1 & | & 3 & 4 & | & 6 & 7 & \cdots & n-1 & 5 \\ \end{array} 
\end{array}\right],\]
where the second part of the $n$-th row is the second part of the second row and the corresponding vertex to the last row is not adjacent to $v_1$ and $v_2$. Therefore, $drn(K_n-K_3)=n-1$.
}\end{proof}
\begin{theorem}\label{p4}
(1) $drn(K_4-P_4)=4$.\\
$drn(K_n-P_4)= n-1$ when $n\in\mathbb{N}_{\geq5}$.
\end{theorem}
\begin{proof}{
(1) Obviously, $K_4-P_4\cong P_4$. Also $P_3$ is an induced subgraph of $K_3-P_4$. So $drn(K_4-P_4)\geq drn(P_3)=4$ by Lemma \ref{inducedsubgraph} and Theorem \ref{thm3}.\\
In addition, the following matrix is the derangement $4$-representation matrix of $K_4-P_4$ and so $drn(K_4-P_4)=4$.
\[L(K_4-P_4)=\left[\begin{array}{cccc}
   {\bf1} & 2 & {\bf3} & {\bf4} \\
   {\bf2} & 1 & 4 & 3 \\
   {\bf1} & 4 & {\bf3} & 2 \\
   {\bf2} & 3 & 1 & {\bf4} \\
   \end{array}\right].\]
(2) Suppose that $G=K_n-P_4$, $V(G)=\{v_i|\ i\in[n]\}$, $v_1\nsim v_2$, $v_2\nsim v_n$ and $v_n\nsim v_3$.
$G-v_n\cong K_{n-1}-K_2$ is an induced subgraph of $K_n-P_4$. So $drn(K_n-P_4)\geq drn(K_{n-1}-K_2)=n-1$ by Lemma \ref{inducedsubgraph} and Theorem \ref{thm3}.\\
For $n\in\{5,6\}$, the following matrices are the derangement $(n-1)$-representation matrices of $K_n-P_4$:
\[L(K_5-P_4)=\left[\begin{array}{cccc}
   1 & 2 & 3 & 4 \\
   {\bf3} & {\bf4} & 2 & {\bf1} \\
   {\bf2} & 1 & {\bf4} & 3 \\
   {\bf3} & {\bf4} & 1 & 2 \\
   {\bf2} & 3 & {\bf4} & {\bf1} \\
   \end{array}\right],
   L(K_6-P_4)=\left[\begin{array}{ccccc}
   {\bf1} & 2 & {\bf3} & {\bf4} & {\bf5} \\
   {\bf2} & {\bf1} & {\bf3} & {\bf5} & 4 \\
   {\bf2} & {\bf1} & 4 & {\bf5} & 3 \\
   {\bf1} & 3 & 2 & {\bf4} & {\bf5} \\
   4 & 5 & 1 & 3 & 2 \\
   3 & 4 & 5 & 2 & 1 \\
   \end{array}\right].\]
Now suppose that $n\geq7$ and $A$ is a latin square of order $n-1$ such that the first three rows of $A$ are \[R_1=\left[\begin{array}{ccccccccc} 1 & 2 & 3 & | & 4 & 5 & \cdots & n-2 & n-1\\ \end{array}\right]\] (corresponding to $v_1$), $R_2=\left[\begin{array}{ccccccccc} 2 & 3 & 1 & | & 5 & 6 & \cdots & n-1 & 4 \\ \end{array}\right]$ (corresponding to $v_2$) and $R_3=\left[\begin{array}{ccccccccc} 3 & 1 & 2 & | & 6 & \cdots & n-1 & 4 & 5 \\ \end{array}\right]$ (corresponding to $v_3$). At first, we replace the second row of $A$ by $R'_2=\left[\begin{array}{ccccccccc} 1 & 2 & 3 & | & 5 & 6 & \cdots & n-1 & 4 \\ \end{array}\right]$ to achieve a derangement $(n-1)$-representation matrix $A'$ of $G-v_n$ in which $v_1$ and $v_2$ are not adjacent. Now the following matrix is a derangement $(n-1)$-representation matrix of $G$:
\[L(K_n-P_4)=\left[\begin{array}{c}
A' \\
\hline
\begin{array}{ccccccccccc} 3 & 1 & 2 & | & 5 & 6 & \cdots & n-1 & 4 \\ \end{array} 
\end{array}\right],\]
where the second part of the $n$-th row is the second part of the second row and the corresponding vertex to the last row ($v_n$) is not adjacent to $v_2$ and $v_3$. Therefore, $drn(K_n-K_3)=n-1$.
}\end{proof}
\begin{theorem}
$drn(K_n-(P_3\cup P_2 ))=n-1$ when $n\in\mathbb{N}_{\geq5}$.
\end{theorem}
\begin{proof}{
Suppose that $G=K_n-(P_3\cup P_2 )$, $V(G)=\{v_i|\ i\in[n]\}$, $v_1\nsim v_2$, $v_n\nsim v_3$ and $v_n\nsim v_4$. $G-v_n\cong K_{n-1}-K_2$ is an induced subgraph of $G$. So $drn(G)\geq drn(K_{n-1}-K_2)=n-1$ by Lemma \ref{inducedsubgraph} and Theorem \ref{thm3}.\\
For $n\in\{5,6\}$, the following matrices are the derangement $(n-1)$-representation matrices of $G$:
\[L(K_5-(P_3\cup P_2 ))=\left[\begin{array}{cccc}
   1 & 2 & {\bf3} & {\bf4} \\
   2 & 1 & {\bf3} & {\bf4} \\
   {\bf3} & {\bf4} & 1 & 2 \\
   4 & 3 & {\bf2} & {\bf1} \\
   {\bf3} & {\bf4} & {\bf2} & {\bf1} \\
   \end{array}\right],
   L(K_6-(P_3\cup P_2 )=\left[\begin{array}{ccccc}
   1 & 2 & {\bf3} & {\bf4} & {\bf5} \\
   2 & 1 & {\bf3} & {\bf4} & {\bf5} \\
   {\bf4} & {\bf3} & 5 & 1 & 2 \\
   5 & 4 & {\bf2} & 3 & {\bf1} \\
   3 & 5 & 1 & 2 & 4 \\
   {\bf4} & {\bf3} & {\bf2} & 5 & {\bf1} \\
      \end{array}\right].\]
Now suppose that $n\geq7$ and $A$ is a latin square of order $n-1$ such that the first four rows of $A$ are $R_1=\left[\begin{array}{cccccccc} 1 & 2 & | & 3 & 4 & \cdots & n-2 & n-1\\ \end{array}\right]$ (corresponding to $v_1$), $R_2=\left[\begin{array}{ccccccccc}  2 & 1 & | & 4 & 5 & \cdots & n-1 & 3 \\ \end{array}\right]$ (corresponding to $v_2$), $R_3=\left[\begin{array}{ccccccccc} 3 & 4 & | & 5 & 6 & \cdots & (n-1) & 1 & 2 \\ \end{array}\right]$ (corresponding to $v_3$) and  $R_4=\left[\begin{array}{ccccccccc} 4 & 3 & | & 6 & \cdots & (n-1) & 1 & 2 & 5 \\ \end{array}\right]$ (corresponding to $v_4$). At first, we replace the second row of $A$ by $R'_2=\left[\begin{array}{ccccccccc}  1 & 2 & | & 4 & 5 & \cdots & n-1 & 3 \\ \end{array}\right]$ to achieve a derangement $(n-1)$-representation matrix $A'$ of $G-v_n$ in which $v_1$ and $v_2$ are not adjacent. Now the following matrix is a derangement $(n-1)$-representation matrix of $G$:
\[L(K_n-(P_3\cup P_2 ))=\left[\begin{array}{c}
A' \\
\hline
\begin{array}{cccccccc} 3 & 4 & | & 6 & \cdots & 1 & 2 & 5 \\ \end{array} 
\end{array}\right],\]
where the second part of the $n$-th row is the second part of  $R_4$ and the corresponding vertex to the last row ($v_n$) is not adjacent to $v_3$ and $v_4$. Therefore, $drn(K_n-(P_3\cup P_2 ))=n-1$.
}\end{proof}
In Theorems \ref{thm2}, \ref{p3}, \ref{p4} and \ref{k3}, we find the derangement representation number of $K_n-P_k$ (for $2\leq k\leq 4$) and $K_n-C_3$. In the following theorems, we obtain upper and lower bounds for the derangement representation number of $K_n-P_k$ and $K_n-C_3$ for the other values of $n$ and $k$.
\begin{theorem}\label{Pkc}
Let $n, k\in\mathbb{N}$. If $n\geq k\geq5$, then $n-\bigg\lceil\dfrac{k}{2}\bigg\rceil+1\leq drn(K_n-P_k)\leq n$.
\end{theorem}
\begin{proof}{
Let $G=K_n-P_k$, $V(G)=\{v_i|\ i\in[n]\}$, $E(G^c)=\{v_iv_{i+1}|\ i\in[k-1]\}$,  $V_0=\{v_{2i}|\ 1\leq i \leq\lfloor\frac{k}{2}\rfloor\}$ and $V_1=\{v_{2i-1}|\ 1\leq i\leq \lceil\frac{k}{2}\rceil\}$.\\
{\bf Case 1.} $k=2t+1$: The subgraph of $G$ induced by $V(G)\setminus V_0$ is a clique of order $n-t$. So $drn(G)\geq \omega(G)=n-t=n-\bigg\lceil\dfrac{k}{2}\bigg\rceil+1$.\\
{\bf Case 2.} $k=2t$: The subgraph of $G$ induced by $(V(G)\setminus V_0)\cup\{v_k\}$ is isomorphic to $K_p-K_2$ where $p=n-t+1$. So $drn(G)\geq drn(K_p-K_2)\geq p=n-\bigg\lceil\dfrac{k}{2}\bigg\rceil+1$.\\
Now we prove that $K_n-P_k$ is derangement $n$-representable. Consider the latin square $L=[l_{i,j}]_{n\times n}$ with $l_{i,j}=j-i+1$ when $j\geq i$ and $l_{i,j}=n+1+j-i$ when $j<i$.
\[L=\left[\begin{array}{ccccccc}
   1 & 2 & 3 & 4 & \cdots & n-1 & n \\
   n & 1 & 2 & 3 & \cdots & n-2 & n-1 \\
   n-1 & n & 1 & 2 & \cdots & n-3 & n-2 \\
   n-2 & n-1 & n & 1 & \cdots & n-4 & n-3 \\
   \vdots & \vdots & \vdots & \vdots & \ddots & \vdots & \vdots \\
   3 & 4 & 5 & 6 & \cdots & 1 & 2 \\
     2 & 3 & 4 & 5 & \cdots & n & 1 \\
   \end{array}\right].\]
   Let $M_{i,j}=\{l_{i,j},l_{i+1,j+1},l_{i+1,j},l_{i+2,j+1}\}$ for any $i,j\in[n]$ (the indices being taken modulo $n$). We have $l_{i,j}=l_{i+1,j+1}$ and $l_{i+1,j}=l_{i+2,j+1}$. If we replace together two cells $l_{i+1,j+1}$ and $l_{i+1,j}$ in $L$, then the result matix is a derangement representation matrix of $K_n-P_2$. We call this change a {\it flip change}. Therefore, to achieve the derangement $n$-representation matrix of $K_n-P_k$, it is enough to do the flip change on $M_{1,1}$, $M_{2,3}$, $\ldots$, and $M_{k-2,2k-5}$ where the indices being taken modulo $n$. In fact, the flip change on $M_{i,2i-1}$, remove the edge $v_iv_{i+1}$ from $E(K_n)$.
}\end{proof}
\begin{example}
As explained in the proof of the Theorem \ref{Pkc}, the following matrx is a derangement representation matrix of $K_{8}-P_6$:
\[L(K_{8}-P_6)=\left[\begin{array}{cccccccc}
{\bf1}&2&3&4&5&6&7&8  \\
{\bf1}&{\bf8}&{\bf2}&3&4&5&6&7  \\
7&{\bf8}&{\bf2}&{\bf1}&{\bf3}&4&5&6  \\
6&7&8&{\bf1}&{\bf3}&{\bf2}&{\bf4}&5  \\
5&6&7&8&1&{\bf2}&{\bf4}&{\bf3}  \\
4&5&6&7&8&1&2&{\bf3}  \\
3&4&5&6&7&8&1&2  \\
2&3&4&5&6&7&8&1  \\ 
  \end{array}\right].\]
\end{example}
\begin{theorem}
Let $n, k\in\mathbb{N}$. If $n\geq k\geq4$, then $n-\bigg\lfloor\dfrac{k}{2}\bigg\rfloor\leq drn(K_n-C_k)\leq n$.
\end{theorem}
\begin{proof}{
Let $G=K_n-C_k$, $V(G)=\{v_i|\ i\in[n]\}$, $E(G^c)=\{v_iv_{i+1}|\ i\in[k-1]\}\cup\{v_kv_1\}$,  $V_0=\{v_{2i}|\ 1\leq i \leq\lfloor\frac{k}{2}\rfloor\}$ and $V_1=\{v_{2i-1}|\ 1\leq i\leq \lceil\frac{k}{2}\rceil\}$.\\
{\bf Case 1.} $k=2t$: The subgraph of $G$ induced by $V(G)\setminus V_0$ is a clique of order $n-t$. So $drn(G)\geq \omega(G)=n-t=n-\bigg\lfloor\dfrac{k}{2}\bigg\rfloor$.\\
{\bf Case 2.} $k=2t+1$: The subgraph of $G$ induced by $V(G)\setminus V_0$ is isomorphic to $K_p-K_2$ where $p=n-t$. So $drn(G)\geq drn(K_p-K_2)\geq p=n-\bigg\lfloor\dfrac{k}{2}\bigg\rfloor$.\\
To prove the upper bound, we consider three cases:\\
{\bf Case 1.} $n=k$: Similar to the proof of Theorem \ref{Pkc}, to achieve the derangement $n$-representation matrix of $K_n-P_n$, it is enough to do the flip change on $M_{1,1}$, $M_{3,3}$, $\ldots$, and $M_{n-1,n-1}$. In fact, the flip change on $M_{i,i}$, remove two edge $v_iv_{i+1}, v_{i+1}v_{i+2}$ from $E(K_n)$.\\
{\bf Case 2.} $n>k=2t\geq4$: Let $A$ and $B$ be latin squares of order $t$ and $n-t$, with symbol sets $[t]$ and $[n]\setminus [t]$, respectively. Also the $i$-th row of a matrix $X$ is denoted by $X_i$. Obviously,
\[L=\left[\begin{array}{c}
A_1\ B_1 \\
\vdots \ \ \ \vdots \\
A_{t}\ B_{t} \\
\end{array}\right]\]
is a latin rectangle of order $t\times n$ and so by use of Hall's Theorem, we can extend $L$ to the following latin rectangle:
\[L'=\left[\begin{array}{c}
A_1\ B_1 \\
\vdots\ \ \vdots \\
A_{t}\ B_{t} \\
 C_{1} \\
 \vdots \\
 C_{n-2t} \\
\end{array}\right]\]
Now easily one can show that the following matrix is a derangement $n$-representation matrix of $K_n-C_k$:
\[L''=\left[\begin{array}{c}
A_1\ B_t \\
A_1\ B_1 \\
A_2\ B_1 \\
A_2\ B_2 \\
A_3\ B_2 \\
\vdots\ \ \ \vdots \\
A_{t}\ B_{t-1} \\
A_{t}\ B_t \\
 C_{1} \\
 \vdots \\
 C_{n-2t} \\
  \end{array}\right].\]
{\bf Case 3.} $n>k=2t-1\geq5$: At first, similar to the previous case we define $A$, $B$ of order $t\times t$ and $t\times (n-t)$, such that $A_1=[1\ 2\ 3\ \cdots\ t]$, $A_t=[2\ 3\ \cdots\ t\ 1]$, and $B_1=[t+1\ \cdots \ n]$. Similarly, we define $L$ (of order $t\times n$) and then we change the symbols $1\rightarrow2\rightarrow t+1 \rightarrow1$ in the first row of $L=[l_{i,j}]$ such that $l_{1,1}=l_{t,1}=2$. Then by use of Hall's Theorem, we can extend $L$ to the row latin rectangle $L'$ (of order $(n-t+1)\times n$) by adding $n-2t+1$ rows $C_1,\ldots,C_{n-2t+1}$ such that $\{c_{i,j}| i\in[n-2t+1]\}\cap\{l_{i,j}| i\in[t]\}=\emptyset$ and $c_{i,j}\neq c_{i',j}\ (i\neq i')$ for any $j\in[n]$. Finally, to achieve a derangement $n$-representation matrix of $K_n-C_k$ we define $L''$ as follows:
\[L''=\left[\begin{array}{c}
A_1\ B_1 \\
A_1\ B_2 \\
A_2\ B_2 \\
A_2\ B_3 \\
A_3\ B_3 \\
\vdots\ \ \ \vdots \\
A_{t-1}\ B_{t-1} \\
A_{t-1}\ B_t \\
A_{t}\ B_t \\
 C_{1} \\
 \vdots \\
 C_{n-2t+1} \\
  \end{array}\right].\]
}\end{proof}
To prove the next theorem, we need an special type of latin squares.
\begin{definition}
A latin square is idempotent if every symbol appears on the main diagonal.
\end{definition}
It was proved in \cite{DesignTheory} that there exists an idempotent latin square of order $n$ for any positive integer $n\neq2$, 
\begin{theorem}\label{thmcycle}
$drn(C_n)\leq\lceil\dfrac{n}{2}\rceil+1$ for any positive integer $n\geq3$.
\end{theorem} 
 \begin{proof}
{Let $V(C_n)=\{v_1,v_2,\ldots,v_{n}\}$, $E(C_n)=\{v_1v_2,v_2v_3,\ldots,v_{n}v_1\}$. We consider two cases:\\
{\bf Case 1.} $n=2k$:\\
Suppose that $V_1=\{v_1,v_3,\ldots ,v_{2k-1}\}$, and $V_0=\{v_2,v_4,\ldots,v_{2k}\}$. There is an idempotent latin square $M=[m_{i,j}]$ of order $k$. Without loss of generality, we assume that $m_{i,i} =i$. Also matrices $N=[n_{i,j}]_{k \times k}$, $L_1$ and $L_0$ are defined as follows.
{\small \[N=\left[\begin{array}{cccccc}
   2 & k+1 &3 &\cdots &k-1 &k \\
   \\
   1  & 3&k+1&\cdots &k-1 &k \\ \\
    1  &2 &4&\cdots &k-1 &k \\ \\
   \vdots &\vdots &\vdots &\ddots &\vdots &\vdots  \\ \\
    1 &2&3&\cdots &k &k+1 \\
   \\
    k+1&2 & 3&\cdots   &k-1 & 1 \\
   \end{array}\right],
L_1=\left[\begin{array}{c|c}
  \begin{array}{c}
   k+1 \\
   \vdots \\
    k+1  \\
    \end{array}
    &M\\
   \end{array}\right],\ 
L_0=\left[\begin{array}{c|c}
  \begin{array}{c}
   1 \\
   2\\
   \vdots \\
    k \\
    \end{array}
    &N\\
   \end{array}\right].\]}
Now if we assign $L_1$ and $L_0$ to the vertices of $V_1$ and $V_0$, respectively then $L=\left[\begin{array}{c} L_1\\ L_0 \\ \end{array}\right]$ is a derangement $(k+1)$-representation of $C_n$.\\
{\bf Case 2.} $n=2k+1$:\\
Suppose that $V_1=\{v_3,v_5,\ldots,v_{2k-1}\}$, $V_0=\{v_2,v_4,\ldots,v_{2k}\}$. We know there is an idempotent latin square $M=[m_{ij}]$ of order $k$. Without loss of generality, we assume that $m_{i,i} =i$. Also matrices $N=[n_{i,j}]_{(k-1)\times k}$, $L_1$, $L_0$, $W_1$ (corresponding to vertex $v_{1}$), and $W_0$ (corresponding to vertex $v_{2k+1}$) are defined as follows.
\[N=\left[\begin{array}{cccccc}
   2 & k+1 &3 &\cdots &k-1 &k \\
   \\
   1  & 3&k+1&\cdots &k-1 &k \\ \\
    1  &2 &4&\cdots &k-1 &k \\ \\
   \vdots &\vdots &\vdots &\ddots &\vdots &\vdots  \\ \\
    1 &2&3&\cdots &k &k+1 \\
   \end{array}\right],\]
   \[L_0=\left[\begin{array}{c|c}
  \begin{array}{cc}
   k+2& k+1 \\
   \vdots \\
    k+2&k+1  \\
    \end{array}
    &M\\
   \end{array}\right], 
   L_1=\left[\begin{array}{c|c}
  \begin{array}{cc}
   1 & k+2 \\
   2& k+2\\
   \vdots \\
    k-1 & k+2 \\
    \end{array}
    &N\\
   \end{array}\right],\]
\[W_0=[k+1,\, 1  \, | \,  k+2 \, , 2 \, ,3 \, ,\cdots,k],\]
\[W_1=[m_{k1},\, k+2  \, | \,  k+1 \, , m_{12} \, ,m_{23} \, ,\cdots,m_{k-1k}].\]
   Now if we assign $L_1$ and $L_0$ to the vertices of $V_1$ and $V_0$, respectively then $L=\left[\begin{array}{c} L_1\\ W_1 \\ L_0 \\ W_0 \\ \end{array}\right]$ is a derangement $(k+2)$-representation of $C_n$.}
\end{proof} 
\begin{example}
As explained in the proof of the Theorem \ref{thmcycle}, the following matrices are the derangement representations matrices of $C_{10}$ and $C_{11}$:
\[L(C_{10})=\left[\begin{array}{c|ccccc}
6 &1&4&2&5&3  \\
6 &4&2&5&3&1  \\
6 &2&5&3&1&4  \\
6 &5&3&1&4&2  \\
6 &3&1&4&2&5  \\
\hline
1 &2&6&3&4&5  \\
2 &1&3&6&4&5  \\
3 &1&2&4&6&5  \\
4 &1&2&3&5&6  \\
5 &6&2&3&4&1  
  \end{array}\right], 
L(C_{11})=\left[\begin{array}{cc|ccccc}
7&6 &1&4&2&5&3  \\
7&6&4&2&5&3&1  \\
7&6 &2&5&3&1&4  \\
7&6&5&3&1&4&2  \\
7&6 &3&1&4&2&5  \\
\hline
6&1&7&2&3&4&5\\
\hline
1&7 &2&6&3&4&5  \\
2&7 &1&3&6&4&5  \\
3&7 &1&2&4&6&5  \\
4&7 &1&2&3&5&6  \\
3&7 &6&4&5&1&2  \\
  \end{array}\right]\]
\end{example}
\begin{theorem}\label{thmpath}
Let $G$ be a path graph of order greater than $4$, then\,
    $ drn(P_n)  \leq \lceil \dfrac{n}{2} \rceil +1.$
\end{theorem} 
\begin{proof}{
Let $V(G)=\{v_1,v_2,\ldots,v_{2k}\}$,$E(G)=\{v_1v_2,v_2v_3,\ldots,v_{n-1}v_{n}\}$. We consider two cases:\\
{\bf Case 1.} $n=2k$: Suppose that $V_1=\{v_1,v_3,\ldots ,v_{2k-1}\}$, and $V_0=\{v_2,v_4,\ldots,v_{2k}\}$. There is an idempotent latin square $M=[m_{i,j}]$ of order $k$. Without loss of generality, we assume that $m_{i,i} =i$. Also matrices $N=[n_{i,j}]_{k \times k}$, $L_1$ and $L_0$ are defined as follows.
{\small \[N=\left[\begin{array}{cccccc}
   2 & k+1 &3 &\cdots &k-1 &k \\
   \\
   1  & 3&k+1&\cdots &k-1 &k \\ \\
    1  &2 &4&\cdots &k-1 &k \\ \\
   \vdots &\vdots &\vdots &\ddots &\vdots &\vdots  \\ \\
    1 &2&3&\cdots &k &k+1 \\
   \\
    1 & 2 & 3 &\cdots   &k-1 & k+1 \\
   \end{array}\right],
L_1=\left[\begin{array}{c|c}
  \begin{array}{c}
   k+1 \\
   \vdots \\
    k+1  \\
    \end{array}
    &M\\
   \end{array}\right],\ 
L_0=\left[\begin{array}{c|c}
  \begin{array}{c}
   1 \\
   2\\
   \vdots \\
    k \\
    \end{array}
    &N\\
   \end{array}\right].\]}
Now if we assign $L_1$ and $L_0$ to the vertices of $V_1$ and $V_0$, respectively then $L=\left[\begin{array}{c} L_1\\ L_0 \\ \end{array}\right]$ is a derangement $(k+1)$-representation of $P_n$.\\
{\bf Case 2.} $n=2k-1$: $P_{2k-1}$ is an induced subgraph of $P_{2k}$. So $drn(P_{2k-1})\leq drn(P_{2k})\leq \lceil\frac{2k}{2}\rceil+1=k+1=\lceil \frac{n}{2}\rceil+1$.
}\end{proof}
\begin{example}
As explained in the proof of the Theorem \ref{thmpath}, the following matrices are the derangement representations matrices of $C_{10}$ and $C_{11}$:
\[L(P_9)=\left[\begin{array}{c|ccccc}
6 &1&4&2&5&3  \\
6 &4&2&5&3&1  \\
6 &2&5&3&1&4  \\
6 &5&3&1&4&2  \\
6 &3&1&4&2&5  \\
\hline
1 &2&6&3&4&5  \\
2 &1&3&6&4&5  \\
3 &1&2&4&6&5  \\
4 &1&2&3&5&6  
  \end{array}\right],
L(P_{10})=\left[\begin{array}{c|ccccc}
6 &1&4&2&5&3  \\
6 &4&2&5&3&1  \\
6 &2&5&3&1&4  \\
6 &5&3&1&4&2  \\
6 &3&1&4&2&5  \\
\hline
1 &2&6&3&4&5  \\
2 &1&3&6&4&5  \\
3 &1&2&4&6&5  \\
4 &1&2&3&5&6  \\
5 &1&2&3&4&6 
    \end{array}\right].\]
\end{example}
\begin{theorem}
Let $n,r\in\mathbb{N}$ and $1<r<n$. Then $drn(K_n-K_r)\leq\max\{n,2r\}$.
\end{theorem} 
\begin{proof}
{Suppose that $V(K_n-K_r)=\{v_1,\ldots,v_n\}$.
At first, suppose that $n\geq 2r$. We give a derangement $n$-representation matrix $L$ of $G$. Let $A=[a_{i,j}]_{r\times r}$ and $B=[b_{i,j}]_{(n-r)\times (n-r)}$ be two latin squares with disjoint sets of symbols $[r]$ and $\{r+1,\ldots,n\}$, respectively. Let $B_1$ be the matrix of order $r\times (n-r)$ that contains the first $r$ rows of $B$. Now we define the first $r$ rows of $L_1$ as follows:
$L_1=\big[\begin{array}{c|c} A & B_1\end{array}\big]$. Then we extend this latin rectangle to a complete latin square $L_2$ of order $n$ by use of Hall's Theorem. Finally, we replace all rows of $A$ with the first row of $A$ in $L_2$ to achieve $L=[l_{i,j}]_{n\times n}$. Now $l_{i,j}=l_{i',j}$ for some $j\in [n]$ if and only if $i,i'\in [r]$ and in other cases, we have $l_{i,j}\neq l_{i',j}$ for all $j\in [n]$. Therefore, $L$ is a derangement $n$-representation matrix of $G$.\\
Now suppose that $n<2r$. In this case, $K_n-K_r$ is an induced subgraph of $K_{2r}-K_r$. Therefore  $drn(K_n-K_r)\leq drn(K_{2r}-K_r)\leq 2r$ by Lemma \ref{inducedsubgraph}.
}\end{proof}
\section{Computational results using Sage }\label{sec4}
We can compute the derangement representation number of graphs of small orders. For example, the following simple code in Sage tests the derangement $4$-representability of the fork graph (See the Figure \ref{fork}):
\begin{figure}[h!]
 \centering
  \includegraphics[width=\textwidth]{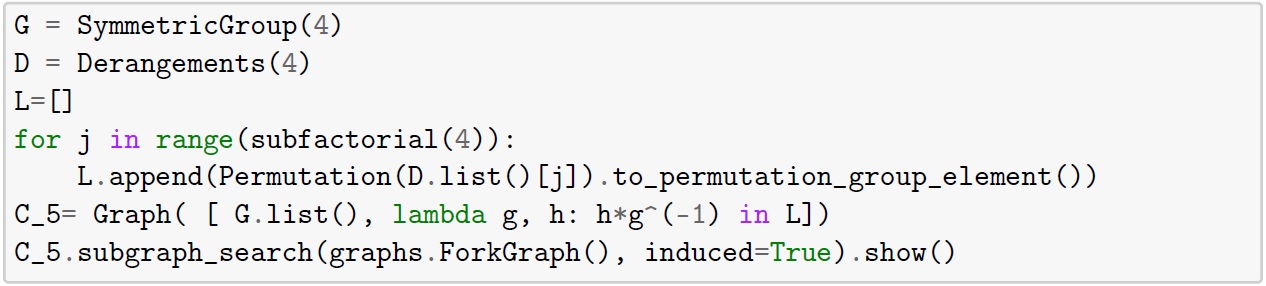}
  \end{figure}
\begin{figure}[h]
\begin{center}
\begin{tikzpicture}[scale=0.5]
 	\tikzset{vertex/.style = {shape=circle,draw, line width=1pt,opacity=1.0}}
  \node[vertex] (x) at (-4,0) {};
  \node[vertex] (y) at (-1,0) {};
  \node[vertex] (z) at (2,0) {};
  \node[vertex] (a) at (5,2) {};
  \node[vertex] (b) at (5,-2) {};
  \foreach \from/\to in {x/y,y/z,z/a,z/b}
  \draw[line width=1pt] (\from) -- (\to);
  \node  at (-4,0.8) {\small{$(1,2,4,3)$}};
  \node  at (-1,-0.8) {\small{$(3,1,2,4)$}};
  \node  at (4.4,0) {\small{$(2,4,3,1)$}};
  \node at (7.2,2) {\small{$(1,4,2,3)$}};
  \node at (7.2,-2) {\small{$(1,2,3,4)$}};
		\end{tikzpicture}
 		\caption{A derangement $4$-representation of fork graph}
 		\label{fork}
 		\end{center}
 		\end{figure}
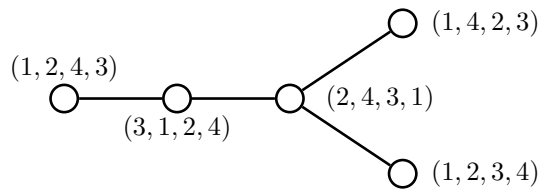
In Table 1, we show the derangement representation number of cycle $C_n$ for any $n\in\{3,\ldots,20\}$:
\[\begin{tabular}{c|c}
    \hline
    $n$ & $drn(C_n)$ \\
    \hline
    3 & 3\\
    4,\ldots, 6 & 4\\
    7,\ldots, 12 & 5\\
    13,\ldots, 20 & 6\\
    \hline
\end{tabular}\]
\begin{center}
Table 1: $drn(C_n)$ for $3\leq n\leq 20$
\end{center}
In Table 2, we show the derangement representation number of path $P_n$ for any $2\leq n\leq 29$:
\[\begin{tabular}{c|c}
    \hline
    $n$ & $drn(P_n)$ \\
    \hline
    2 & 2\\
    3,\ldots, 8 & 4\\
    9,\ldots, 14 & 5\\
    15,\ldots, 22 & 6\\
    23,\ldots, 29 &7\\
    \hline
\end{tabular}\]
\begin{center}
Table 2: $drn(P_n)$ for $2\leq n\leq 29$
\end{center}
In Table 3, we find the derangement representation number of complete bipartite graph $K_{r,s}$ when $r,s\in [10]$:
\[\begin{tabular}{cc|c}
    \hline
    $r$ & $s$ & $drn(K_{r,s})$ \\
    \hline
    1 &1 & 2\\
    1,2 & 2,3 & 4\\
    3 & 3 & 5\\
    1,\ldots,4 & 4 & 5\\
    1,\ldots,5 & 5 & 5\\
    1,\ldots,6 & 6 & 5\\
    1,\ldots,4 & 7 & 5\\
    5,\ldots,7 & 7 & 6\\
    1,\ldots,4 & 8 & 5\\
    5,\ldots,8 & 8 & 6\\
    1,2 & 9 & 5\\
    3,\ldots,6 & 9 & 6\\
    7,\ldots,9 & 9 & 7\\
    1 & 10 & 5\\
    2,\ldots,6 & 10 & 6\\
    7,\ldots,10 & 10 & 7\\
    \hline
\end{tabular}\]
\begin{center}
Table 3: $drn(K_{r,s})$ for $r,s\in[10]$
\end{center}
Suppose that $Cay^{not}(i)$ represents the number of graphs of order $i$ which are not induced subgraphs of $Cay(S_i, D_i)$. For some small values of $i$, we compute $Cay^{not}(i)$ using SAGE,  with the following code. The results for $i\in[7]$ have been shown in Table 4.
\begin{figure}[h!]
 \centering
  \includegraphics[width=\textwidth]{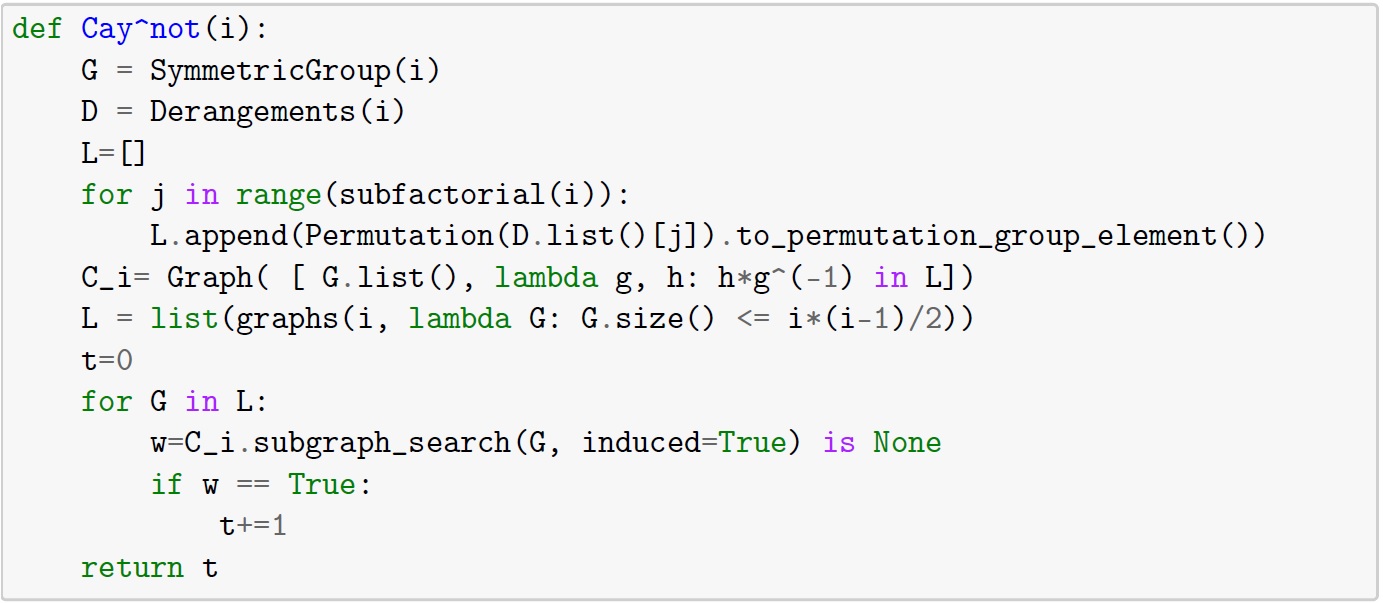}
  \end{figure}     
\[\begin{tabular}{c|c}
    \hline
    $i$ & $Cay^{not}(i)$ \\
    \hline
    1 & 0\\
    2 & 1\\
    3 & 2\\
    4,\ldots,7 & 0\\
    \hline
\end{tabular}\]
\begin{center}
Table 4: $Cay^{not}(i)$ for $i\in[7]$
\end{center}
Using the code "$G$.subgraph$\_$search($H$, induced=True).show()" in Sage, we can identify an induced subgraph of $G$ that is isomorphic to $H$. For instance, in the following code we identify the graphs of order $6$ with derangement representation number $6$. In fact, $K_6$, $K_6-K_2$, $K_6-2K_2$ and $K_6-K_3$ are the only graphs of order six with derangement representation number six:  (See Figure \ref{n6}).
\begin{figure}[h!]
 \centering
  \includegraphics[width=\textwidth]{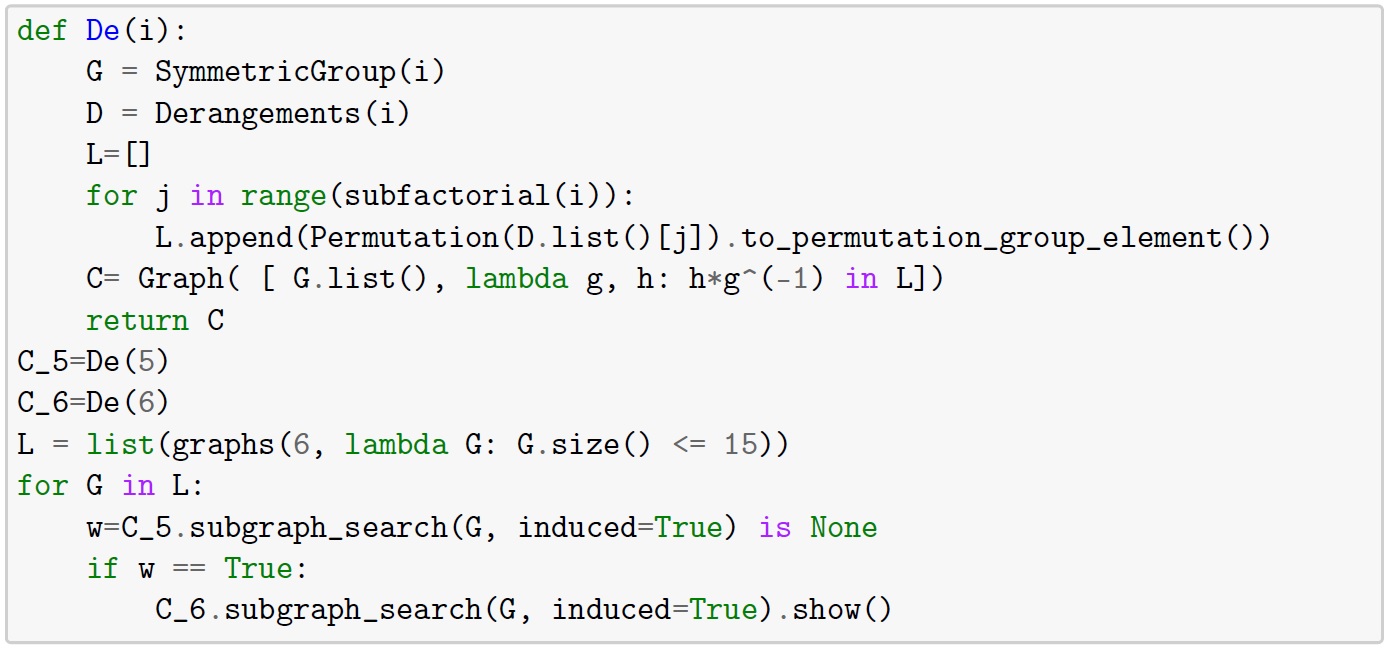}
  \end{figure}
 \begin{figure}[h!]
  \begin{center}
 	\begin{tikzpicture}[scale=0.4]
 	\tikzset{vertex/.style = {shape=circle,draw, line width=1pt,opacity=1.0}}
  \node[vertex] (x) at (-3,-2) {};
  \node[vertex] (y) at (0,-0.3) {};
  \node[vertex] (z) at (3,-2){};
  \node[vertex] (a) at (3,-6) {};
  \node[vertex] (b) at (0,-7.7) {};
  \node[vertex] (c) at (-3,-6) {};
  \foreach \from/\to in {x/y,y/z,y/b,z/a,z/x,z/c,z/b,a/b,a/x,b/c,b/y,b/x,c/x}
  \draw[line width=1pt] (\from) -- (\to);
  \node  at (-6.2,-2) {\small{$(2,1,4,3,6,5)$}};
  \node  at (0,0.6) {\small{$(1,2,6,4,5,3)$}};
  \node  at (6.3,-2){\small{$(5,6,1,2,3,4)$}};
  \node  at (6.3,-6) {\small{$(1,2,3,5,4,6)$}};
  \node at (0,-8.6) {\small{$(3,4,5,6,1,2)$}};
  \node  at (-6.2,-6) {\small{$(1,2,3,4,5,6)$}};  
		\end{tikzpicture}
		\begin{tikzpicture}[scale=0.4]
 	\tikzset{vertex/.style = {shape=circle,draw, line width=1pt,opacity=1.0}}
  \node[vertex] (x) at (-3,-2) {};
  \node[vertex] (y) at (0,-0.3) {};
  \node[vertex] (z) at (3,-2){};
  \node[vertex] (a) at (3,-6) {};
  \node[vertex] (b) at (0,-7.7) {};
  \node[vertex] (c) at (-3,-6) {};
  \foreach \from/\to in {x/y,y/z,y/b,y/a,y/c,z/a,z/x,z/c,a/b,a/c,b/c,b/x,c/x}
  \draw[line width=1pt] (\from) -- (\to);
  \node  at (-6.2,-2) {\small{$(2,6,4,1,3,5)$}};
  \node  at (0,0.6) {\small{$(4,3,6,5,2,1)$}};
  \node  at (6.3,-2) {\small{$(3,4,5,6,1,2)$}};
  \node  at (6.3,-6)  {\small{$(2,1,4,3,6,5)$}};
  \node at (0,-8.6) {\small{$(1,2,3,4,5,6)$}};
  \node  at (-6.2,-6) {\small{$(3,5,1,6,4,2)$}};  
		\end{tikzpicture}
		\begin{tikzpicture}[scale=0.4]
 	\tikzset{vertex/.style = {shape=circle,draw, line width=1pt,opacity=1.0}}
 \node[vertex] (x) at (-3,-2) {};
  \node[vertex] (y) at (0,-0.3) {};
  \node[vertex] (z) at (3,-2){};
  \node[vertex] (a) at (3,-6) {};
  \node[vertex] (b) at (0,-7.7) {};
  \node[vertex] (c) at (-3,-6) {};
  \foreach \from/\to in {x/y,y/z,y/b,y/a,y/c,z/a,z/b,z/x,a/x,a/b,a/c,b/c,b/x,c/x}
  \draw[line width=1pt] (\from) -- (\to);
  \node  at (-6.4,-2) {{$(5,6,1,2,3,4)$}};
  \node  at (0,0.6) {{$(3,4,5,6,1,2)$}};
  \node  at (6.5,-2) {{$(2,5,4,1,6,3)$}};
  \node  at (6.5,-6) {{$(1,2,3,4,5,6)$}};
  \node at (0,-8.6) {{$(4,3,6,5,2,1)$}};
  \node  at (-6.4,-6) {{$(2,1,4,3,6,5)$}};  
		\end{tikzpicture}
		 		\caption{Derangement $6$-representations of $K_6-K_3$, $K_6-2K_2$ and $K_6-K_2$}
		 		\label{n6}
 		\end{center}
 		 \end{figure}
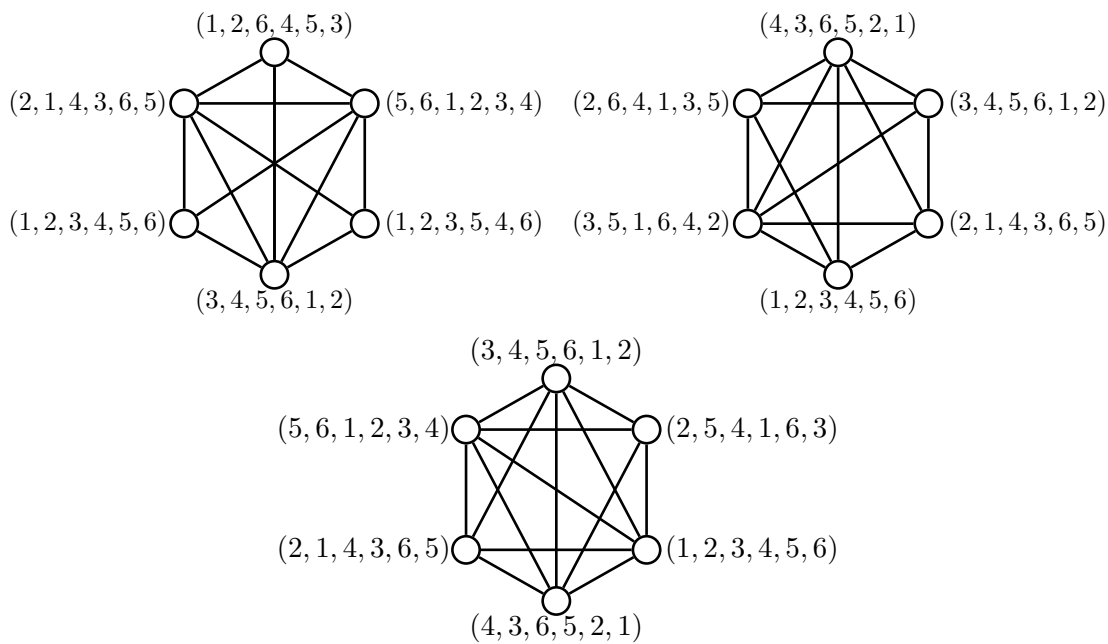
\section{Further questions}\label{sec5}
As outlined in Sections \ref{sec3} and \ref{sec4}, especially in Table 4, we see that the derangement representation number of any graph which is reviewed in this paper is less than its order, except a finite number of graphs ($\mathcal{F}=\{\overline{K_2}, \overline{K_3}, P_3\}$). So we propose the following conjecture:
\begin{conjecture}\label{conj1}
$drn(G)\leq p(G)$ for any finite graph $G$ except the graphs of the family $\mathcal{F}$.
\end{conjecture}
We can easily show that the complete graph $K_n$ is derangement $k$-representable for all $k\geq n$. In fact, any latin rectangle of order $n\times k$ is a derangement representation matrix of $K_n$. There is a similar result for the graph $\overline{K_n}$.
\begin{problem}
Characterise all graphs $G$ which are derangement $k$-representable for any positive integer $k\geq drn(G)$.
\end{problem}

\end{document}